\documentclass[final,3p,times]{elsarticle}

 \usepackage{graphics}
\usepackage{mathrsfs}
\usepackage{graphicx}
\usepackage{epsfig}
\usepackage{subfig}
\usepackage{natbib}
\usepackage[nomarkers]{endfloat}
\usepackage{amssymb}
\usepackage{amsthm}
\usepackage{amsmath}

\biboptions{sort&compress}
\newtheorem{theorem}{Theorem}[section]
\newtheorem{lemma}{Lemma}[section]
\newtheorem{remark}{Remark}[section]
\newtheorem{proposition}{Propositon}[section]

\allowdisplaybreaks[4]
\journal{}

\begin{document}

\begin{frontmatter}

\title{Critical traveling waves in a diffusive disease model\tnoteref{label1}}
\tnotetext[label1]{This research was supported in part by NSF of China under the Grant 11731014.}

\author[]{Jiangbo Zhou\corref{cor1}}
\cortext[cor1]{Corresponding author.}
\ead{ujszjb@126.com}

\author[]{Haimei Xu}

\author[]{Jingdong Wei}

\author[]{Liyuan Song}

\address{Nonlinear Scientific Research Center, Faculty of Science, Jiangsu University, Zhenjiang, Jiangsu 212013, People's Republic of China}

\begin{abstract}
In this paper, the existence of a non-trivial, positive and bounded critical traveling wave solution of a diffusive disease model, whose reaction system has infinity many equilibria, is obtained for the first time. This gives an affirmative answer to an open problem left in [X. Wang, H. Wang, J. Wu, Traveling waves of diffusive predator-prey systems: disease outbreak propagation, Discrete Contin. Dyn. Syst. Ser. A 32 (2012) 3303-3324]. Our result shows that the critical traveling wave in this model is a mixed of front and pulse type.
\end{abstract}



\end{frontmatter}

\section{Introduction}
\setcounter {equation}{0}

Recently, Wang et al. \cite{Wang-Wang-Wu-2012} considered a diffusive disease model
\begin{equation}
\label{eq1.1}
\left\{ {\begin{array}{l}
\displaystyle  S_t(x,t)=d_1S_{xx}(x,t)-\frac{\beta S(x,t)I(x,t)}{S(x,t)+I(x,t)},\\
\displaystyle  I_t(x,t)=d_2I_{xx}(x,t)+\frac{\beta S(x,t)I(x,t)}{S(x,t)+I(x,t)}-\gamma I(x,t),\\
\displaystyle  R_t(x,t)=d_3R_{xx}(x,t)+\gamma I(x,t),
\end{array}} \right.
\end{equation}
where $S(x,t)$, $I(x,t)$ and $R(x,t)$ are the densities of susceptible, infective and removed individuals at location $x$ and time $t$, respectively. The parameters $d_{i}>0 \; (i=1,2,3)$ denote the spatial motility of each class, $\beta>0$ is the transmission coefficient and $\gamma>0$ refers to the recovery rate. Since the first two equations in (\ref{eq1.1}) form a closed system,  they only studied the subsystem of (\ref{eq1.1})
\begin{equation}
\label{eq1.2}
\left\{ {\begin{array}{l}
\displaystyle  S_t(x,t)=d_1S_{xx}(x,t)-\frac{\beta S(x,t)I(x,t)}{S(x,t)+I(x,t)},\\
\displaystyle  I_t(x,t)=d_2I_{xx}(x,t)+\frac{\beta S(x,t)I(x,t)}{S(x,t)+I(x,t)}-\gamma I(x,t).
 \end{array}} \right.
\end{equation}
The traveling wave profile of (\ref{eq1.2}) is given by the system of second-order ordinary differential equations:
\begin{equation}
\label{eq1.3}
\left\{ {\begin{array}{l}
\displaystyle cS'(\xi)=d_1 S''(\xi)-\frac{\beta S(\xi)I(\xi)}{S(\xi)+I(\xi)}, \\
\displaystyle  cI'(\xi)=d_2 I''(\xi)+\frac{\beta S(\xi)I(\xi)}{S(\xi)+I(\xi)}-\gamma I(\xi),
 \end{array}} \right.
\end{equation}
where $\xi=x+ct$ is the wave variable and $c$ is the wave speed. They proved that if the basic reproduction number $\mathcal{R}_0:=\beta/\gamma>1$, then there exists a critical wave speed $c^{*}:=2\sqrt{d_2(\beta-\gamma)}$ such that for each $c>c^{*}$, system (\ref{eq1.1}) admits a non-trivial and non-negative traveling wave solution $(S(\xi),I(\xi))$ which satisfies the asymptotic boundary conditions
\begin{equation}
\label{eq1.4}
S(-\infty)=S_{-\infty}, \quad S(\infty)=S_{\infty}<S_{-\infty}, \quad I(\pm \infty)=0,
\end{equation}
where $S_{-\infty}>0$ is a given constant and $S_{\infty}\geq0$ is some existing constant. They also showed that if $\mathcal{R}_0\leq1$ or $c<c^{*}$, then system (\ref{eq1.1}) has no non-trivial and non-negative traveling wave solutions. When $\mathcal{R}_0>1$ and $c=c^{*}$, the existence of traveling wave solutions for (\ref{eq1.2}) was left as an open problem in \cite{Wang-Wang-Wu-2012}.  Very recently, Wang et al. \cite{Wang-Nie-Wu-2017} employed a limiting argument to solve this problem. Unfortunately, their proof of the non-triviality of the limit is problematic because they used the integrability of $I$ over $\mathbb{R}$ before they hadn't found the asymptotic boundary of the critical traveling wave (see \cite{Wang-Nie-Wu-2017}, Lemma 2.1 and p. 416). In view of this, the problem on the existence of critical traveling wave solutions for (\ref{eq1.2}) remains open.

Traveling waves, especially the ones with the minimal/critical speed, play an important role in the study of disease transmission. The minimal wave speed of a disease is a crucial threshold to predict whether the disease propagates and how fast it invades. Until recently, there has been a number of work devoting to the existence of super-critical and critical traveling waves in diffusive epidemic models \cite{Bai-Wu-2015, Ducrot-Langlais-Magal-2012, Ducrot-Magal-2011,Gan-Xu-Yang-2011, Hethcote-2000, Hosono-Ilyas-1994, Tian-Yuan-2017, Wang-Wang-2016, Wang-Nie-Wu-2017, Wang-Wang-Wu-2012, Wang-Wu-2010, Xu-2014, Xu-2017, Xu-Ai-2016, Zhen-2018-2, Zhou-2016, Zhang-Wang-2014, Wang-Wu-Liu-2012,Zhao-Wang-Ruan-2017, Zhan-Huo-Liu-Sun-2009, Thieme-Zhao-2003, Ding-Huang-Kansakar-2013, Ai-Albashaireh-2014, Li-Li-Lin-2015, Weng-Zhao-2006, Zhao-Wang-2004, Fu-2016, Zhao-Wang-2016, Zhao-Wang-Ruan-2018, Tian-Yuan-2017-2, Wu-Peng-2013}. However, all the investigation of critical traveling waves are limited to the models whose reaction systems have one disease-free equilibrium and one endemic equilibrium. Note that the reaction system of (\ref{eq1.2}) has infinity many equilibria, then the shooting method coupled with LaSalle's invariance principle \cite{Ai-Albashaireh-2014} and the method of dynamic systems \cite{Zhan-Huo-Liu-Sun-2009, Ding-Huang-Kansakar-2013} are not applicable to study the existence of critical traveling wave solutions for (\ref{eq1.2}). Also the standard limiting argument \cite{Thieme-Zhao-2003, Weng-Zhao-2006, Zhao-Wang-2004, Li-Li-Lin-2015, Zhao-Wang-2016, Zhao-Wang-Ruan-2018, Wu-Peng-2013} fails since the $I-$component in (\ref{eq1.2}) is not monotone.

In the present paper, motivated by \cite{Fu-2016, Wang-Wang-Wu-2012},  we intend to solve the problem on the existence of critical traveling wave solutions for (\ref{eq1.2}). The main result is stated as follows.
\begin{theorem}
\label{the1.1}
Suppose that $\mathcal{R}_0>1$ and $c=c^{*}$. Then system (\ref{eq1.2}) admits a non-trivial, positive and bounded traveling wave solution $(S_*(\xi),I_*(\xi))$ satisfying the asymptotic boundary conditions (\ref{eq1.4}) and the following properties.
\begin{itemize}
\item[(1)] $S_*(\xi)$ is strictly decreasing in $\mathbb{R}$  and
\[
0<  S_*(\xi)<S_{-\infty} \quad \textmd{for}\quad \xi\in\mathbb{R}.
\]
\item[(2)] $I_*(\xi)= \mathcal{O}(-\xi e^{\lambda^*\xi})$ as $\xi \rightarrow -\infty$ and
\[
0< I_*(\xi)<\min \bigg\{\frac{(\beta-\gamma)S_{-\infty}}{\gamma}, \frac{2c^*(S_{-\infty}-S_\infty)}{\sqrt{(c^*)^{2}+4d_2\gamma}+c^*}\bigg\} \quad for \;\; \xi\in\mathbb{R}.
\]
\item[(3)]
\[
\int_{-\infty}^{\infty} I_*(\xi)d\xi=\frac{\beta}{\gamma}\int_{-\infty}^{\infty} \frac{S_*(\xi)I_*(\xi)}{S_*(\xi)+I_*(\xi)}d\xi=\frac{c^*(S_{-\infty}-S_\infty)}{\gamma}.
\]
\end{itemize}
\end{theorem}
\begin{remark}
\begin{itemize}
\item[(1)] Theorem \ref{the1.1} shows that the critical traveling wave in (\ref{eq1.2}) is a mixed of front and pulse type.
\item[(2)] To our knowledge, Theorem \ref{the1.1} gives the first result on the existence of positive critical traveling wave solutions for diffusive disease models whose reaction systems have infinity many equilibria.
\item[(3)] Wang et al. \cite{Wang-Wang-Wu-2012} obtained that the super-critical traveling wave solutions of (\ref{eq1.2}) are non-negative and the $S$-component is monotonically decreasing in $\mathbb{R}$. While Theorem \ref{the1.1} indicates that the critical traveling wave solution of (\ref{eq1.2}) is positive and the $S$-component is strictly decreasing in $\mathbb{R}$. In fact, one can apply the method adopted by us to improve the corresponding results in  \cite{Wang-Wang-Wu-2012}.
\end{itemize}
\end{remark}

The remainder of this paper is organized as follows. Section 2-Section 4 are devoting to the proof of Theorem \ref{the1.1}.  In Section 2, the existence of a critical traveling wave solution $(S_*(\xi), I_*(\xi))$ of (\ref{eq1.2}) is established by Schauder's fixed point theorem. In Section 3, the asymptotic boundary of $(S_*(\xi), I_*(\xi))$ at minus infinity is found with the aid of squeeze theorem. While the asymptotic boundary of $S_*(\xi)$ at plus infinity is derived from the monotonicity and boundedness of $S_*(\xi)$ in $\mathbb{R}$ and the asymptotic boundary of $I_*(\xi)$ at plus infinity is shown by the integrability of $I_*(\xi)$  and the boundedness of $I_*'(\xi)$ in $\mathbb{R}$. In Section 4, the positivity and upper bound of $(S_*(\xi), I_*(\xi))$ are obtained via the strong maximum principle and auxiliary function method.

\section{Existence of critical traveling wave solution}
\setcounter {equation}{0}
In this section, we will establish the existence of a critical traveling wave solution for (\ref{eq1.2}) by Schauder's fixed point theorem.

Linearizing the second equation in (\ref{eq1.3}) at $(S_{-\infty},0)$,  we have
\[
d_2 I''(\xi)-cI'(\xi)+(\beta-\gamma)I(\xi)=0.
\]
Letting $I(\xi)=e^{\lambda \xi}$, we then get the characteristic equation
\begin{equation}
\label{eq2.1}
\Delta(\lambda,c):=d_{2}\lambda^{2}-c\lambda +\beta-\gamma=0.
\end{equation}
Obviously, when $c=c^{*}=2\sqrt{d_{2}(\beta-\gamma)}$, equation (\ref{eq2.1}) has a unique real root
\[
\lambda^{*}=\frac{c^{*}}{2d_{2}}=\sqrt{\frac{\beta-\gamma}{d_{2}}}.
\]

Now we define four non-negative, continuous and bounded functions in $\mathbb{R}$.
\[
\begin{split}
&\overline{S}(\xi):=S_{-\infty},\\
&\overline{I}(\xi):=\min\bigg\{-L_{1}\xi e^{\lambda^*\xi}, M\bigg\}=\left\{ {\begin{array}{l}-L_{1}\xi e^{\lambda^*\xi},\quad\quad\quad\quad\quad\quad\;\xi\leq\xi_1,\\
M,\quad\quad\quad\quad\quad\quad\quad\;\;\;\;\;\;\;\xi>\xi_1,
 \end{array}} \right.\\
&\underline{S}(\xi):=\max \bigg\{S_{-\infty}\bigg(1-\frac{1}{\epsilon}e^{\epsilon\xi}\bigg),0\bigg\}=\left\{ {\begin{array}{l}S_{-\infty}\big(1-\frac{1}{\epsilon}e^{\epsilon\xi}\big),\quad \quad\;\xi\leq\xi_2,\\
0,\quad\quad\quad\quad\quad\quad\quad\;\;\xi>\xi_2,
 \end{array}} \right.\\
&\underline{I}(\xi):=\max \bigg\{[-L_{1}\xi -L_2(-\xi)^{\frac{1}{2}}]e^{\lambda^*\xi},0\bigg\}=\left\{ {\begin{array}{l}[-L_{1}\xi -L_2(-\xi)^{\frac{1}{2}}]e^{\lambda^*\xi},\;\;\;\xi\leq\xi_3,\\
0,\quad\quad\quad\quad\quad\quad\quad\;\;\;\;\;\;\; \xi>\xi_3,
 \end{array}} \right.
\end{split}
\]
where
\[
M:=\frac{(\beta-\gamma)S_{-\infty}}{\gamma},\quad L_{1}:=eM\lambda^{*},\quad \xi_{1}:=-\frac{1}{\lambda^{*}},\quad \xi_{2}:=\frac{1}{\epsilon}\ln\epsilon,\quad \xi_{3}:=-\bigg(\frac{L_2}{L_1}\bigg)^2,
\]
$\epsilon $ and $L_{2}$ are two positive constants to be determined.
\begin{lemma}
\label{lem2.1}
There exist a sufficiently small constant $\epsilon>0$  and a large enough constant $L_{2}>0$, such that the functions $\overline{S}(\xi)$, $\overline{I}(\xi)$, $\underline{S}(\xi)$ and $\underline{I}(\xi)$ satisfy
\begin{equation}
\label{eq2.2}
\begin{split}
c^{*}\overline{I}'(\xi)\geq d_{2}\overline{I}''(\xi)+\frac{\beta\overline{S}(\xi)\overline{I}(\xi)}{\overline{S}(\xi)+\overline{I}(\xi)}-\gamma\overline{I}(\xi)\quad \textmd{for}\quad \xi\neq\xi_{1},
\end{split}
\end{equation}
\begin{equation}
\label{eq2.3}
\begin{split}
-\beta\overline{I}(\xi)\geq-d_{1}\underline{S}''(\xi)+c^{*}\underline{S}'(\xi)\quad \textmd{for} \quad  \xi< \xi_2
\end{split}
\end{equation}
and
\begin{equation}
\label{eq2.4}
\begin{split}
\frac{\beta \underline{S}(\xi)\underline{I}(\xi)}{\underline{S}(\xi)+\underline{I}(\xi)}-\gamma\underline{I}(\xi)\geq-d_{2}\underline{I}''(\xi)+c^{*}\underline{I}'(\xi)\quad \textmd{for}\quad  \xi<\xi_{3}.
\end{split}
\end{equation}
\end{lemma}
\begin{proof}
\emph{Proof of (\ref{eq2.2})}. When $\xi<\xi_{1}$, $\overline{I}(\xi)=-L_{1}\xi e^{\lambda^{*}\xi}$. It follows that
\[
\begin{split}
d_{2}\overline{I}''(\xi)-c^{*}\overline{I}'(\xi)+\frac{\beta\overline{S}(\xi)\overline{I}(\xi)}{\overline{S}(\xi)+\overline{I}(\xi)}-\gamma\overline{I}(\xi)&\leq d_{2}\overline{I}''(\xi)-c^{*}\overline{I}'(\xi)+(\beta-\gamma)\overline{I}(\xi)\\
&=\big(d_{2}({\lambda^*})^{2}-c^{*}\lambda^{*}+\beta-\gamma\big)(-L_{1}\xi e^{\lambda^{*}\xi})\\
&=0 \quad \textmd{for} \quad \xi<\xi_{1}.
\end{split}
\]
When $\xi>\xi_{1}$, $\overline{I}(\xi)=M=(\beta-\gamma)S_{-\infty}/\gamma$. We have
\[
\begin{split}
d_{2}\overline{I}''(\xi)-c^{*}\overline{I}'(\xi)+\frac{\beta\overline{S}(\xi)\overline{I}(\xi)}{\overline{S}(\xi)+\overline{I}(\xi)}-\gamma\overline{I}(\xi)&
\leq \frac{\beta S_{-\infty}M}{S_{-\infty}+M}-\gamma M\\
&=0\quad \textmd{for} \quad \xi>\xi_{1}.
\end{split}
\]

\emph{Proof of (\ref{eq2.3})}. Let $\epsilon>0$ be sufficiently small such that $\epsilon^{-1}\ln\epsilon <-(\lambda^{*})^{-1}$, i.e., $\xi_{2}<\xi_{1}$. When $\xi<\xi_{2}$,
\[
\overline{I}(\xi)=-L_{1}\xi e^{\lambda^{*}\xi}
\]
and
\[
\underline{S}(\xi)=S_{-\infty}\bigg(1-\frac{1}{\epsilon}e^{\epsilon\xi}\bigg).
\]
Then inequality (\ref{eq2.3}) is equivalent to
\[
\begin{split}
\beta L_{1}\xi e^{\lambda^{*}\xi}&\geq d_{1}S_{-\infty}\epsilon e^{\epsilon \xi}-c^{*}S_{-\infty}e^{\epsilon \xi}\\
&=S_{-\infty}(d_{1}\epsilon -c^{*})e^{\epsilon \xi} \;\; \textmd{for}\;\;\xi<\xi_{2},
\end{split}
\]
that is,
\[S_{-\infty}(c^{*}-d_{1}\epsilon )\geq-\beta L_{1}\xi e^{(\lambda^{*}-\epsilon )\xi} \;\; \textmd{for}\;\;\xi<\xi_{2},
\]
which holds for sufficiently small $\epsilon\in (0, \min\{c^{*}/d_1, \lambda^{*}\})$.

\emph{Proof of (\ref{eq2.4})}. Let $L_2>0$ be large enough such that $L_{2}>L_{1}\sqrt{-\epsilon ^{-1}\ln\epsilon}$. When $\xi<\xi_{3}$,
\[
\underline{S}(\xi)=S_{-\infty}\bigg(1-\frac{1}{\epsilon}e^{\epsilon\xi}\bigg)
\]
and
\[
\underline{I}(\xi)=\big[-L_{1}\xi-L_{2}(-\xi)^{\frac{1}{2}}\big]e^{\lambda^{*}\xi}.
\]
Then to prove (\ref{eq2.4}) is to show
\[
-\frac{\beta\underline{I}^{2}(\xi)}{\underline{S}(\xi)+\underline{I}(\xi)}\geq-d_{2}\underline{I}''(\xi)+c^{*}\underline{I}'(\xi)-(\beta-\gamma)\underline{I}(\xi)\;\; \textmd{for }\;\xi<\xi_{3},
\]
that is,
\[
\frac{-\beta(-L_{1}\xi-L_{2}(-\xi)^{\frac{1}{2}}\big)^{2}e^{2\lambda^{*}\xi}}{S_{-\infty}\bigg(1-\frac{1}{\epsilon}e^{\epsilon\xi}\bigg)+(-L_{1}\xi-L_{2}(-\xi)^{\frac{1}{2}})e^{\lambda^{*}\xi}}\geq-\frac{d_{2}L_{2}}{4}(-\xi)^{-\frac{3}{2}}e^{\lambda^{*}\xi}\;\; \textmd{for }\;\xi<\xi_{3}.
\]
It suffices to verify that
\begin{equation}
\label{eq2.5}
\begin{split}
d_{2}L_2 S_{-\infty}\bigg(1-\frac{1}{\epsilon}e^{\epsilon\xi}\bigg)
&\geq 4\beta(-L_{1}\xi)^{2}(-\xi)^{\frac{3}{2}}e^{\lambda^{*}\xi}\\
&=4\beta L_{1}^{2}(-\xi)^{\frac{7}{2}}e^{\lambda^{*}\xi}\;\;\; \textmd{for }\;\;\xi<\xi_{3}.
\end{split}
\end{equation}
Let $g(\xi):=\beta L_{1}^{2} (-\xi )^{\frac{7}{2}}e^{\lambda^{*}\xi}$ for $\xi\leq 0$. Then $g(\xi)\in C(-\infty,0]$ and $g(-\infty)=f(0)=0$ and there exists a constant $\tilde{C}>0$ independent of $L_2$, such that
\[
g(\xi)\leq \tilde{C} ~~~\textmd{for}~~~\xi\leq 0.
\]
Therefore, inequality (\ref{eq2.5}) holds for large enough $L_2>0$. This ends the proof.
\end{proof}

We rewrite system (\ref{eq1.3}) in an equivalent form:
\[
\left\{ {\begin{array}{l}
\displaystyle (d_1S''(\xi)-c^{*}S'(\xi)-\beta_{1}S(\xi))+\bigg(\beta_{1}S(\xi)-\frac{\beta S(\xi)I(\xi)}{S(\xi)+I(\xi)}\bigg)=0,\\
\displaystyle (d_2I''(\xi)-c^{*}I'(\xi)-\beta_{2}I(\xi))+\bigg((\beta_{2}-\gamma)I(\xi)+\frac{\beta S(\xi)I(\xi)}{S(\xi)+I(\xi)}\bigg)=0.
 \end{array}} \right.
\]
Here we choose $\beta_{1}>\beta$ and $\beta_{2}>\gamma$, such that
\[
a_1(S, I):=\beta_{1}S-\frac{\beta SI}{S+I}
\]
is non-decreasing in $S$ and non-increasing in $I$, while
\[
a_2(S, I):=(\beta_{2}-\gamma)I+\frac{\beta SI}{S+I}
\]
is non-decreasing in $S$ and $I$.

Let $\lambda_{i}^{\pm}$ be the roots of equation $d_{i}\lambda^{2}-c^{*}\lambda-\beta_{i}=0$ $(i=1,2)$,
then we obtain
\[
\lambda_{i}^{\pm}=\frac{c^{*}\pm\sqrt{(c^{*})^{2}+4d_{i}\beta_{i}}}{2d_{i}} \;\;\; \textmd{for} \;\; i=1,2.
\]

Now we introduce a subset of $C(\mathbb{R},\mathbb{R}^{2})$
\[
\Gamma:=\big\{(S(\cdot),I(\cdot))\in B_\mu(\mathbb{R},\mathbb{R}^2):\underline{S}(\xi)\leq S(\xi)\leq S_{-\infty},\;\underline{I}(\xi)\leq I(\xi)\leq \overline{I}(\xi)\big\},
\]
where
\[
\begin{split}
B_\mu(\mathbb{R},\mathbb{R}^2)=&\max\bigg\{\Psi=(\phi(\cdot),\varphi(\cdot))\in C(\mathbb{R},\mathbb{R}^2)\; \bigg| \;\sup_{\xi\in \mathbb{R}}|\phi(\xi)|e^{-\mu|\xi|}<\infty,\;\;\sup_{\xi\in \mathbb{R}}|\varphi(\xi)|e^{-\mu|\xi|}<\infty\bigg\}
\end{split}
\]
equipped with the norm
\[|\Psi|_\mu:=\max\bigg\{\sup_{\xi\in \mathbb{R}}|\phi(\xi)|e^{-\mu|\xi|}, \sup_{\xi\in \mathbb{R}}|\varphi(\xi)|e^{-\mu|\xi|}\bigg\}.
\]
Here the constant $\mu$ satisfies $0<\mu<\min\{-\lambda_{1}^{-},-\lambda_{2}^{-}\}$. Then $B_\mu(\mathbb{R},\mathbb{R}^2)$ is a Banach space with the norm $|\cdot|_\mu$ and $\Gamma$ is uniformly bounded with respect to the norm $|\cdot|_\mu$.

Define a function $f:\Gamma\mapsto C(\mathbb{R})$
\[
\begin{split}
&f(S,I)(\xi):=\left\{ {\begin{array}{l}\displaystyle \frac{\beta S(\xi)I(\xi)}{S(\xi)+I(\xi)},\quad S(\xi)I(\xi)\neq 0,\\
0,\quad\quad \quad\quad\;\;\;\;S(\xi)I(\xi)= 0,
 \end{array}} \right.\\
\end{split}
\]
and an operator
\[
F=(F_{1}(S,I)(\xi),F_{2}(S,I)(\xi)):\Gamma\mapsto C(\mathbb{R},\mathbb{R}^{2}),
\]
where
\[
F_{i}(S,I)(\xi)=\frac{1}{\Lambda_{i}}\bigg(\int_{-\infty}^{\xi}e^{\lambda_{i}^{-}(\xi-y)}h_{i}(S,I)(y)dy+\int_{\xi}^{\infty}e^{\lambda_{i}^{+}(\xi-y)}h_{i}(S,I)(y)dy\bigg)\;\;\;\;\textmd{for}\;\; i=1,2,
\]
\[
\Lambda_{i}=d_{i}(\lambda_{i}^{+}-\lambda_{i}^{-})=\sqrt{(c^{*})^{2}+4d_{i}\beta_{i}},
\]
\[
h_{1}(S,I)(\xi)=\beta_{1}S(\xi)-f(S,I)(\xi),
\]
and
\[h_{2}(S,I)(\xi)=(\beta_{2}-\gamma)I(\xi)+f(S,I)(\xi).
\]

 Next we will prove that the operator $F$ satisfies the conditions of Schauder's fixed point theorem.
\begin{lemma}
\label{lem2.2}
The operator $F:\Gamma\mapsto\Gamma$ is completely continuous with respect to $|\cdot|_\mu$ in $B_\mu(\mathbb{R},\mathbb{R}^2)$.
\end{lemma}
\begin{proof}
We will divide the proof into the following three steps.

 \emph{Step 1. We prove that $F$ maps $\Gamma$ into $\Gamma$}.

For any $((S(\xi),I(\xi))\in \Gamma$,  we only need to show
\[
\underline{S}(\xi)\leq F_{1}(S,I)(\xi)\leq S_{-\infty}\;\;\textmd{and}\;\;\underline{I}(\xi)\leq F_{2}(S,I)(\xi)\leq \overline{I}(\xi)  ~~~\textmd{for}~~~\xi\in\mathbb{R}.
\]
Let  $(S(\xi),I(\xi))\in \Gamma$, then we obtain
\[
\begin{split}
F_{1}(S,I)(\xi)&=\frac{1}{\Lambda_{1}}\bigg[\int_{-\infty}^{\xi}e^{\lambda_{1}^{-}(\xi-y)}(\beta_{1}S(y)-f(S,I)(y))dy+\int_{\xi}^{\infty}e^{\lambda_{1}^{+}(\xi-y)}(\beta_{1}S(y)-f(S,I)(y))dy\bigg]\\
&\leq\frac{1}{\Lambda_{1}}\bigg(\int_{-\infty}^{\xi}e^{\lambda_{1}^{-}(\xi-y)}\beta_{1}S(y)dy+\int_{\xi}^{\infty}e^{\lambda_{1}^{+}(\xi-y)}\beta_{1}S(y)dy\bigg)\\
&\leq\frac{\beta_{1}S_{-\infty}}{\Lambda_{1}}\bigg(\int_{-\infty}^{\xi}e^{\lambda_{1}^{-}(\xi-y)}dy+\int_{\xi}^{\infty}e^{\lambda_{1}^{+}(\xi-y)}dy\bigg)\\
&=\frac{\beta_{1}S_{-\infty}}{\Lambda_{1}}\bigg(\frac{1}{\lambda_{1}^{+}}-\frac{1}{\lambda_{1}^{-}}\bigg)\\
&=S_{-\infty}\;\; \;\textmd{for }\;\; \xi\in \mathbb{R}.
\end{split}
\]
When $\xi<\xi_{2}$, we deduce from Lemma \ref{lem2.1} and the definition of $\underline{S}(\xi)$ that
\begin{equation}
\label{eq2.6}
\begin{split}
F_{1}(S,I)(\xi)&=\frac{1}{\Lambda_{1}}\bigg[\int_{-\infty}^{\xi}e^{\lambda_{1}^{-}(\xi-y)}(\beta_{1}S(y)-f(S,I)(y))dx+\int_{\xi}^{\infty}e^{\lambda_{1}^{+}(\xi-y)}(\beta_{1}S(y)-f(S,I)(y))dy\bigg]\\
&\geq\frac{1}{\Lambda_{1}}\bigg[\int_{-\infty}^{\xi}e^{\lambda_{1}^{-}(\xi-y)}(\beta_{1}\underline{S}-\beta\overline{I})(y)dx+\int_{\xi}^{\infty}e^{\lambda_{1}^{+}(\xi-y)}(\beta_{1}\underline{S}-\beta\overline{I})(y)dy\bigg]\\
&\geq\frac{1}{\Lambda_{1}}\bigg[\int_{-\infty}^{\xi}e^{\lambda_{1}^{-}(\xi-y)}(\beta_{1}\underline{S}-d_{1}\underline{S}''+c^{*}\underline{S}')(y)dy+\int_{\xi}^{\xi_{2}}e^{\lambda_{1}^{+}(\xi-y)}(\beta_{1}\underline{S}-d_{1}\underline{S}''+c^{*}\underline{S}')(y)dy\bigg]\\ &=\frac{1}{\Lambda_{1}}\bigg\{\int_{-\infty}^{\xi}e^{\lambda_{1}^{-}(\xi-y)}\bigg[\beta_{1}S_{-\infty}+\frac{S_{-\infty}}\epsilon(d_{1}\epsilon^{2}-c^{*}\epsilon-\beta_{1})e^{\epsilon y}\bigg]dy\\
&\quad+\int_{\xi}^{\xi_{2}}e^{\lambda_{1}^{+}(\xi-y)}\bigg[\beta_{1}S_{-\infty}+\frac{S_{-\infty}}\epsilon(d_{1}\epsilon^{2}-c^{*}\epsilon-\beta_{1})e^{\epsilon y}\bigg]dy\bigg\}\\
&=S_{-\infty}\bigg(1-\frac{1} \epsilon e^{\epsilon \xi}\bigg)+\frac{d_{1}\epsilon S_{-\infty}}{\Lambda_{1}}e^{\lambda_1^{+}(\xi-\xi_{2})}\\
&\geq S_{-\infty}\bigg(1-\frac{1} \epsilon e^{\epsilon \xi}\bigg).
\end{split}
\end{equation}
When $\xi>\xi_{2}$, we get
\begin{equation}
\label{eq2.7}
\begin{split}
F_{1}(S,I)(\xi)&\geq\frac{1}{\Lambda_{1}}\bigg[\int_{-\infty}^{\xi}e^{\lambda_{1}^{-}(\xi-y)}(\beta_{1}-\beta)S(y)dy+\int_{\xi}^{\infty}e^{\lambda_{1}^{+}(\xi-y)}(\beta_{1}-\beta)S(y)dy\bigg]\\
&\geq\frac{1}{\Lambda_{1}}\bigg[\int_{-\infty}^{\xi}e^{\lambda_{1}^{-}(\xi-y)}(\beta_{1}-\beta)\underline{S}(y)dy+\int_{\xi}^{\infty}e^{\lambda_{1}^{+}(\xi-y)}(\beta_{1}-\beta)\underline{S}(y)dy\bigg]\\
&\geq0.
\end{split}
\end{equation}
Combining (\ref{eq2.6}), (\ref{eq2.7}) and the continuities of $\underline{S}(\xi)$ and $F_{1}(S,I)(\xi)$ in $\mathbb{R}$,
we conclude that
\[
F_{1}(S,I)(\xi)\geq\underline{S}(\xi) \;\; \;\textmd{for }\;\;\xi\in\mathbb{R}.
\]
When $\xi < \xi_{3}$, in view of Lemma \ref{lem2.1} and the definition of $\underline{I}(\xi)$, we infer that
\begin{equation}
\label{eq2.8}
\begin{split}
F_{2}(S,I)(\xi)&=\frac{1}{\Lambda_{2}}\bigg\{\int_{-\infty}^{\xi}e^{\lambda_{2}^{-}(\xi-y)}((\beta_{2}-\gamma)I(y)+f(S,I)(y))dy
+\int_{\xi}^{\infty}e^{\lambda_{2}^{+}(\xi-y)}((\beta_{2}-\gamma)I(y)+f(S,I)(y))dy\bigg\}\\
&\geq\frac{1}{\Lambda_{2}}\bigg\{\int_{-\infty}^{\xi}e^{\lambda_{2}^{-}(\xi-y)}\bigg[(\beta_{2}-\gamma)\underline{I}+\frac{\beta\underline{S}\;\underline{I}}{\underline{S}+\underline{I}}\bigg](y)dy+\int_{\xi}^{\xi_{3}}e^{\lambda_{2}^{+}(\xi-y)}\bigg[(\beta_{2}-\gamma)\underline{I}+\frac{\beta\underline{S}\;\underline{I}}{\underline{S}+\underline{I}}\bigg](y)dy\bigg\}\\
&\geq\frac{1}{\Lambda_{2}}\bigg[\int_{-\infty}^{\xi}e^{\lambda_{2}^{-}(\xi-y)}(\beta_{2}\underline{I}-d_{2}\underline{I}''+c^{*}\underline{I}')(y)dy+\int_{\xi}^{\xi_{3}}e^{\lambda_{2}^{+}(\xi-y)}(\beta_{2}\underline{I}-d_{2}\underline{I}''+c^{*}\underline{I}')(y)dy\bigg]\\
&=\frac{1}{\Lambda_{2}}\bigg\{\int_{-\infty}^{\xi}e^{\lambda_{2}^{-}(\xi-y)}\bigg[(d_{2}{(\lambda^{*}})^{2}-c^{*}\lambda^{*}-\beta_{2})L_{1}ye^{\lambda^{*}y}+L_{2}(d_{2}{(\lambda^{*}})^{2}-c^{*}\lambda^{*}-\beta_{2})(-y)^{\frac{1}{2}}e^{\lambda^{*}y}\\
&\quad-\frac{1}{4}d_{2}L_{2}(-y)^{-\frac{3}{2}}e^{\lambda^{*}y}\bigg]dy\\
&\quad+\int_{\xi}^{\xi_{3}}e^{\lambda_{2}^{+}(\xi-y)}\bigg[(d_{2}{(\lambda^{*}})^{2}-c^{*}\lambda^{*}-\beta_{2})L_{1}ye^{\lambda^{*}y}+L_{2}(d_{2}{(\lambda^{*}})^{2}-c^{*}\lambda^{*}-\beta_{2})(-y)^{\frac{1}{2}}e^{\lambda^{*}y}\\
 &\quad-\frac{1}{4}d_{2}L_{2}(-y)^{-\frac{3}{2}}e^{\lambda^{*}y}\bigg]dy\bigg\}\\
&=[-L_{1}\xi-L_{2}(-\xi)^{\frac{1}{2}}]e^{\lambda^{*}\xi}+\frac{L_{1}d_{2}}{2\Lambda_{2}}e^{\lambda^{*}\xi_{3}}e^{\lambda_{2}^{+}(\xi-\xi_{3})}\\
&\geq[-L_{1}\xi-L_{2}(-\xi)^{\frac{1}{2}}]e^{\lambda^{*}\xi}.\\
\end{split}
\end{equation}
When $\xi>\xi_{3}$, $\underline{I}(\xi)=0$. Since $I(\xi)\geq\underline{I}(\xi)$ and $\beta_{2}>\gamma$, we obtain
\[
\begin{split}
F_{2}(S,I)(\xi)&\geq\frac{1}{\Lambda_{2}}\int_{-\infty}^{\xi_{3}}e^{\lambda_{2}^{-}(\xi-y)}\bigg[\frac{\beta\underline{S}\;\underline{I}}{\underline{S}+\underline{I}}+(\beta_{2}-\gamma)\underline{I}\bigg](y)dy\\
&\geq0\;\; \textmd{for} \;\; \xi>\xi_{3},
\end{split}
\]
which together with (\ref{eq2.8}) and the continuities of $\underline{I}(\xi)$ and $F_{2}(S,I)(\xi)$ in $\mathbb{R}$ yields that
\[
F_{2}(S,I)(\xi)\geq\underline{I}(\xi)\;\; \;\textmd{for }\;\;\xi\in\mathbb{R}.
\]
When $\xi<\xi_{1}$, $\overline{I}(\xi)=-L_{1}\xi e^{\lambda^*\xi}$. One can deduce from Lemma \ref{lem2.1} and  the definition of $\overline{I}(\xi)$  that
\begin{equation}
\label{eq2.9}
\begin{split}
 F_{2}(S,I)(\xi)&=\frac{1}{\Lambda_{2}}\bigg\{\int_{-\infty}^{\xi}e^{\lambda_{2}^{-}(\xi-y)}((\beta_{2}-\gamma)I(y)+f(S,I)(y))dy
+\int_{\xi}^{\infty}e^{\lambda_{2}^{+}(\xi-y)}((\beta_{2}-\gamma)I(y)+f(S,I)(y))dy\bigg\}\\
&\leq\frac{1}{\Lambda_{2}}\bigg\{\int_{-\infty}^{\xi}e^{\lambda_{2}^{-}(\xi-y)}\bigg[(\beta_{2}-\gamma)\overline{I}+\frac{\beta\overline{S}\;\overline{I}}{\overline{S}+\overline{I}}\bigg](y)dx+\int_{\xi}^{\xi_{1}}e^{\lambda_{2}^{+}(\xi-y)}\bigg[(\beta_{2}-\gamma)\overline{I}+\frac{\beta\overline{S}\;\overline{I}}{\overline{S}+\overline{I}}\bigg](y)dy\\
&\quad+\int_{\xi_1}^{\infty}e^{\lambda_{2}^{+}(\xi-y)}\bigg[(\beta_{2}-\gamma)\overline{I}+\frac{\beta\overline{S}\;\overline{I}}{\overline{S}+\overline{I}}\bigg](y)dy\bigg\}\\
&\leq\frac{1}{\Lambda_{2}}\bigg\{\int_{-\infty}^{\xi}e^{\lambda_{2}^{-}(\xi-y)}(-d_{2}\overline{I}''+c^{*}\overline{I}'+\beta_{2}\overline{I})(y)dy+\int_{\xi}^{\xi_{1}}e^{\lambda_{2}^{+}(\xi-y)}(-d_{2}\overline{I}''+c^{*}\overline{I}'+\beta_{2}\overline{I})(y)dy\\
&\quad+\int_{\xi_{1}}^{\infty}e^{\lambda_{2}^{+}(\xi-y)}(-d_{2}\overline{I}''+c^{*}\overline{I}'+\beta_{2}\overline{I})(y)dy\bigg\}\\
&=\frac{1}{\Lambda_{2}}\bigg\{\int_{-\infty}^{\xi}e^{\lambda_{2}^{-}(\xi-y)}\big[d_{2}(\lambda^{*})^{2}-c^{*}\lambda^{*}-\beta_{2}\big]L_{1}ye^{\lambda^{*}y}dy+\int_{\xi}^{\xi_{1}}e^{\lambda_{2}^{+}(\xi-y)}\big[d_{2}(\lambda^{*})^{2}-c^{*}\lambda^{*}-\beta_{2}\big]L_{1}xe^{\lambda^{*}y}dy\\
&\quad+\int_{\xi_{1}}^{\infty}e^{\lambda_{2}^{+}(\xi-y)}\beta_{2}Mdy\bigg\}\\
&=\frac{1}{\Lambda_{2}}L_{1}d_{2}(\lambda_{2}^{-}-\lambda_{2}^{+})\xi e^{\lambda^{*}\xi}\\
&=-L_{1}\xi e^{\lambda^{*}\xi}\;\;\;\textmd{for}\;\;\; \xi<\xi_{1}.
\end{split}
\end{equation}
When $\xi>\xi_{1}$, $\overline{I}(\xi)=M$. It follows from Lemma \ref{lem2.1} and  the definition of $\overline{I}(\xi)$  that
\begin{equation}
\label{eq2.10}
\begin{split}
 F_{2}(S,I)(\xi)&\leq\frac{1}{\Lambda_{2}}\bigg\{\int_{-\infty}^{\xi_{1}}e^{\lambda_{2}^{-}(\xi-y)}\bigg[(\beta_{2}-\gamma)\overline{I}+\frac{\beta\overline{S}\;\overline{I}}{\overline{S}+\overline{I}}\bigg](y)dy+\int_{\xi_{1}}^{\xi}e^{\lambda_{2}^{-}(\xi-y)}\bigg[(\beta_{2}-\gamma)\overline{I}+\frac{\beta\overline{S}\;\overline{I}}{\overline{S}+\overline{I}}\bigg](y)dy\\
&\quad+\int_{\xi}^{\infty}e^{\lambda_{2}^{+}(\xi-y)}\bigg[(\beta_{2}-\gamma)\overline{I}+\frac{\beta\overline{S}\;\overline{I}}{\overline{S}+\overline{I}}\bigg] dy\bigg\}\\
&\leq\frac{1}{\Lambda_{2}}\bigg\{\int_{-\infty}^{\xi_{1}}e^{\lambda_{2}^{-}(\xi-y)}(-d_{2}\overline{I}''+c^{*}\overline{I}'+\beta_{2}\overline{I})(y)dy+\int_{\xi_{1}}^{\xi}e^{\lambda_{2}^{-}(\xi-y)}(-d_{2}\overline{I}''+c^{*}\overline{I}'+\beta_{2}\overline{I})(y)dy\\
&\quad+\int_{\xi}^{\infty}e^{\lambda_{2}^{+}(\xi-y)}(-d_{2}\overline{I}''+c^{*}\overline{I}'+\beta_{2}\overline{I})(y)dy\bigg\}\\
&=\frac{1}{\Lambda_{2}}\bigg\{\int_{-\infty}^{\xi_{1}}e^{\lambda_{2}^{-}(\xi-y)}\big[d_{2}(\lambda^{*})^{2}-c^{*}\lambda^{*}-\beta_{2}\big]L_{1}ye^{\lambda^{*}y}dy+\int_{\xi_{1}}^{\xi}e^{\lambda_{2}^{-}(\xi-y)}\beta_{2}Mdy+\int_{\xi}^{\infty}e^{\lambda_{2}^{+}(\xi-y)}\beta_{2}Mdy\bigg\}\\
&=M\;\;\;\textmd{for}\;\;\; \xi>\xi_{1}.
\end{split}
\end{equation}
By (\ref{eq2.9}), (\ref{eq2.10}) and the continuities of $\overline{I}(\xi)$ and $F_{2}(S,I)(\xi)$  in $\mathbb{R}$, we obtain that
\[
F_{2}(S,I)(\xi)\leq\overline{I}(\xi)\;\; \;\textmd{for }\;\;\xi\in\mathbb{R}.
\]

\emph{Step 2. We prove that $F$ is continuous respect to the norm $|\cdot|_\mu$ in $B_\mu(\mathbb{R},\mathbb{R}^2)$}.

For any $(S_{1},I_{1}), (S_{2},I_{2})\in\Gamma$,  we have
\[
\begin{split}
&|F_{1}(S_{1},I_{1})(\xi)-F_{1}(S_{2},I_{2})(\xi)|e^{-\mu|\xi|}\\
&=\frac{1}{\Lambda_{1}}\bigg[e^{-\mu|\xi|}\int_{-\infty}^{\xi}e^{\lambda_{1}^{-}(\xi-y)}(h_{1}(S_{1},I_{1})(y))-h_{1}(S_{2},I_{2})(y))dy+e^{-\mu|\xi|}\int_{\xi}^{+\infty}e^{\lambda_{2}^{+}(\xi-y)}(h_{1}(S_{1},I_{1})(y))-h_{1}(S_{2},I_{2})(y))dy\bigg]\\
&\leq\frac{1}{\Lambda_{1}}\bigg[\int_{-\infty}^{\xi}e^{-\mu|\xi|}e^{\mu|y|}e^{\lambda_{1}^{-}(\xi-y)}(\beta_{1}+\beta)(|S_{1}(y)-S_{2}(y)|e^{-\mu|y|}+|I_{1}(y)-I_{2}(y)|e^{-\mu|y|})dy\\
&\quad+\int_{\xi}^{+\infty}e^{-\mu|\xi|}e^{\mu|y|}e^{\lambda_{1}^{+}(\xi-y)}(\beta_{1}+\beta)(|S_{1}(y)-S_{2}(y)|e^{-\mu|y|}+|I_{1}(y)-I_{2}(y)|e^{-\mu|y|})dy\bigg]\\
&\leq\frac{\beta_{1}+\beta}{\Lambda_{1}}\bigg(\int_{-\infty}^{\xi}e^{\mu|\xi-y|}e^{\lambda_{1}^{-}(\xi-y)}dx+\int_{\xi}^{+\infty}e^{\mu|\xi-y|}e^{\lambda_{1}^{+}(\xi-y)}dy\bigg)(|S_{1}-S_{2}|_\mu+|I_{1}-I_{2}|_\mu)\\
&\leq\frac{\beta_{1}+\beta}{\Lambda_{1}}\bigg(-\frac{1}{\lambda_{1}^{-}+\mu}+\frac{1}{\lambda_{1}^{+}-\mu}\bigg)(|S_{1}-S_{2}|_\mu+|I_{1}-I_{2}|_\mu),
\end{split}
\]
which implies that $F_{1}$ is continuous with respect to the norm $|\cdot|_\mu$ in $B_\mu(\mathbb{R},\mathbb{R}^2)$.  The continuity of $F_{2}$ can be shown analogously.

\emph{Step 3. We prove that $F$ is compact with respect to the norm $|\cdot|_\mu$ in $B_\mu(\mathbb{R},\mathbb{R}^2)$.}

The proof can be carried out by the similar argument as that in [16, Lemma 3.5] or [21, Lemma 3.7]. For completeness, we give the details here. For any $(S,I)\in\Gamma$, we deduce that
\begin{equation}
\label{eq2.11}
\begin{split}
|F_{1}'(S,I)(\xi)|&=\bigg|\frac{\lambda_{1}^{-}}{\Lambda_{1}}\int_{-\infty}^{\xi}e^{\lambda_{1}^{-}(\xi-y)}h_{1}(S,I)(y)dy+\frac{\lambda_{1}^{+}}{\Lambda_{1}}\int_{\xi}^{\infty}e^{\lambda_{1}^{+}(\xi-y)}h_{1}(S,I)(y)dy\bigg|\\
&\leq -\frac{\lambda_{1}^{-}}{\Lambda_{1}}\int_{-\infty}^{\xi}e^{\lambda_{1}^{-}(\xi-y)}\beta S(y)dy+\frac{\lambda_{1}^{+}}{\Lambda_{1}}\int_{\xi}^{\infty}e^{\lambda_{1}^{+}(\xi-y)}\beta S(y)dy\\
&\leq -\frac{\beta_{1}S_{-\infty}\lambda_{1}^{-}}{\Lambda_{1}}\int_{-\infty}^{\xi}e^{\lambda_{1}^{-}(\xi-y)}dy+\frac{\beta_{1}S_{-\infty}\lambda_{1}^{+}}{\Lambda_{1}}\int_{\xi}^{\infty}e^{\lambda_{1}^{+}(\xi-y)}dy\\
&=\frac{2\beta_{1}S_{-\infty}}{\Lambda_{1}}
\end{split}
\end{equation}
and
\begin{equation}
\label{eq2.12}
\begin{split}
|F_{2}'(S,I)(\xi)|&=\bigg|\frac{\lambda_{2}^{-}}{\Lambda_{2}}\int_{-\infty}^{\xi}e^{\lambda_{2}^{-}(\xi-y)}h_{2}(S,I)(y)dy+\frac{\lambda_{2}^{+}}{\Lambda_{2}}\int_{\xi}^{\infty}e^{\lambda_{2}^{+}(\xi-y)}h_{2}(S,I)(y)dy\bigg|\\
&\leq -\frac{\lambda_{2}^{-}}{\Lambda_{2}}\int_{-\infty}^{\xi}e^{\lambda_{2}^{-}(\xi-y)}(\beta+\beta_{2}-\gamma)I(y)dy+\frac{\lambda_{2}^{+}}{\Lambda_{2}}\int_{\xi}^{\infty}e^{\lambda_{2}^{+}(\xi-y)}(\beta+\beta_{2}-\gamma)I(y)dy \\
&\leq  -\frac{(\beta+\beta_{2}-\gamma)M_{1}\lambda_{2}^{-}}{\Lambda_{2}}\int_{-\infty}^{\xi}e^{\lambda_{2}^{-}(\xi-y)}dy+\frac{(\beta+\beta_{2}-\gamma)M_{1}\lambda_{2}^{+}}{\Lambda_{2}}\int_{\xi}^{\infty}e^{\lambda_{1}^{+}(\xi-y)}dy\\
&=\frac{2(\beta+\beta_{2}-\gamma)M_{1}}{\Lambda_{2}}.
\end{split}
\end{equation}
Hence $F=(F_{1}, F_{2})$ is equi-continuous on any compact interval in $\mathbb{R}$. Moreover,  one can infer from the argument in Step 2 that
\begin{equation}
\label{eq2.13}
\begin{split}
|F_{1}(S,I)(\xi)|\leq S_{-\infty}\;\; \;\textmd{and}\;\;\; |F_{2}(S,I)(\xi)|\leq M\;\;\;\textmd{for}\;\;\;\xi \in \mathbb{ R}.
\end{split}
\end{equation}
Hence, for any $\varepsilon>0$, there exists a positive integer $N$ such that
\begin{equation}
\label{eq2.14}
(|F_{1}(S,I)(\xi)|+|F_{2}(S,I)(\xi)|)e^{-\mu|\xi|}\leq ( S_{-\infty}+M)e^{-\mu|\xi|}<\epsilon\; \;\;\textmd{for}\;\; \;|\xi|>N.
\end{equation}
Applying (\ref{eq2.11})-(\ref{eq2.13}) and  the Arzel\`{a}-Ascoli theorem, one can find finite elements in $F(\Gamma)$ such that there exist a finite $\epsilon$-net of $F(\Gamma)(\xi)$  in sense of supremum norm if they are restricted on $ [-N, N]$, which is also a finite $\epsilon$-net of $F(\Gamma)(\xi)$ on $(-\infty, \infty)$ in sense of the norm $|\cdot|_\mu$ (by (\ref{eq2.14})). The compactness of $F$ with respect to the norm $|\cdot|_\mu$ then follows.  This completes the proof.
\end{proof}

It is easy to see that $\Gamma$ is nonempty, bounded, closed and convex. Then from Lemma \ref{lem2.2} and Schauder's fixed point theorem we conclude that $F$ admits a fixed point $(S_{*}(\xi),I_{*}(\xi))\in\Gamma$, which is a solution of the system
\begin{equation}
\label{eq2.15}
\left\{ {\begin{array}{l}
\displaystyle c^*S_*'(\xi)=d_1 S_*''(\xi)-\frac{\beta S_*(\xi)I_*(\xi)}{S_*(\xi)+I_*(\xi)}, \\
\displaystyle  c^*I_*'(\xi)=d_2 I_*''(\xi)+\frac{\beta S_*(\xi)I_*(\xi)}{S_*(\xi)+I_*(\xi)}-\gamma I_*(\xi).
 \end{array}} \right.
\end{equation}

Therefore, we have the following existence result for (\ref{eq1.2}).
\begin{proposition}
\label{the2.1}
Suppose that $\mathcal{R}_0>1$ and $c=c^{*}$. Then (\ref{eq1.2}) has a non-negative traveling wave solution $(S_*(\xi),I_*(\xi))$ satisfying
\begin{equation}
\label{eq2.16}
\begin{split}
 \underline{S}(\xi)\leq S_{*}(\xi)\leq S_{-\infty}\;\;and \;\;\underline{I}(\xi)\leq I_{*}(\xi)\leq\overline{I}(\xi)
\end{split}
\end{equation}
for $\xi\in \mathbb{R}$.
\end{proposition}

\section{Asymptotic boundary of the critical traveling wave solution}
\setcounter {equation}{0}
In this section, we will derive the asymptotic boundary of the critical traveling wave solution for (\ref{eq1.2}), whose existence is established in Section 2.

\begin{proposition}
\label{pro3.1}
Suppose that $(S_*(\xi),I_*(\xi))$ is a critical traveling wave solution of (\ref{eq1.2}), then it has the following asymptotic boundary
\begin{itemize}
\item[(1)] $S_{*}(\xi)\rightarrow S_{-\infty},~~I_{*}(\xi)\rightarrow 0,~~I_*(\xi)= \mathcal{O}(-\xi e^{\lambda^*\xi}) ~~~   \textmd{as}~~~ \xi\rightarrow-\infty$,
\item[(2)] $\lim_{\xi\rightarrow \infty} S_{*}(\xi) := S_{\infty}\;\;\textmd{exists}, \;\; S_{\infty}<S_{*}(-\infty)=S_{-\infty}\;\; and\;\; I_{*}(\xi)\rightarrow 0~~   \textmd{as}~~~ \xi\rightarrow\infty.$
\end{itemize}
\end{proposition}
\begin{proof}
By squeeze theorem and (\ref{eq2.16}), we obtain
\begin{equation}
\label{eq3.1}
S_{*}(\xi)\rightarrow S_{-\infty},~~I_{*}(\xi)\rightarrow 0 ~~\textmd{and} ~~I_*(\xi)= \mathcal{O}(-\xi e^{\lambda^*\xi}) ~~~   \textmd{as}~~~ \xi\rightarrow-\infty.
\end{equation}
Moreover, it follows from  (\ref{eq2.15}), (\ref{eq2.16}) and L'Hospital's rule that
\begin{equation}
\label{eq3.2}
\begin{split}
S_{*}'(\xi)\rightarrow0,~~~I_{*}'(\xi)\rightarrow0~~~~\textmd{as~}~~~\xi\rightarrow\pm\infty.
\end{split}
\end{equation}
By the first equation in (\ref{eq2.15}) one can derive
\begin{equation}
\label{eq3.3}
\begin{split}
\bigg(e^{-\frac{c^{*}}{d_{1}}\xi}S_{*}'(\xi)\bigg)'=e^{-\frac{c^{*}}{d_{1}}\xi}\frac{\beta S_{*}(\xi)I_{*}(\xi)}{d_{1}(S_{*}(\xi)+I_{*}(\xi))}.
\end{split}
\end{equation}
Integrating (\ref{eq3.3}) over $(\xi, \infty)$, using (\ref{eq3.2}) and noting that $S_{*}(\xi)$ and $I_{*}(\xi)$ are non-negative, we then get
\begin{equation}
\label{eq3.4}
S_{*}'(\xi)=-e^{-\frac{c^{*}}{d_{1}}\xi}\int_{\xi}^{\infty}e^{-\frac{c^{*}}{d_{1}}v}\frac{\beta S_{*}(v)I_{*}(v)}{d_{1}(S_{*}(v)+I_{*}(v))}dv\leq 0 \;\; \textmd{for} \;\; \xi\in \mathbb{R},
\end{equation}
which implies that $S_{*}(\xi)$ is monotonically decreasing in $\mathbb{R}$. Additionally, $S_{*}(\xi)$ is bounded in $\mathbb{R}$ (by (\ref{eq2.16})). Then it follows that
\begin{equation}
\label{eq3.5}
\lim_{\xi\rightarrow \infty}S_{*}(\xi):=S_{\infty}\;\;\textmd{exists}.
\end{equation}
Furthermore, since $S_*(\xi)\geq \underline{S}(\xi)>0$ and $I_*(\xi)\geq \underline{I}(\xi)>0$ for $\xi<\min\{\xi_2, \xi_3\}$, we have from (\ref{eq3.4}) that
\begin{equation}
\label{eq3.6}
S_*'(\xi)<0\;\; \textmd{for} \;\;\xi<\min\{\xi_2, \xi_3\}.
\end{equation}
Then one can infer that
\[
S_{\infty}<S_{*}(-\infty)=S_{-\infty}.
\]

Now integrating the first equation in (\ref{eq2.15}) over $\mathbb{R}$ and using (\ref{eq3.1}), (\ref{eq3.2}) and (\ref{eq3.5}) give that
\begin{equation}
\label{eq3.7}
\begin{split}
\int_{-\infty}^{\infty}\frac{S_{*}(v)I_{*}(v)}{S_{*}(v)+I_{*}(v)}dv=\frac{c^{*}(S_{-\infty}-S_{\infty})}{\beta}.
\end{split}
\end{equation}
Another integration of the second equation in (\ref{eq2.15}) from $-\infty$ to $\xi$ leads to
\begin{equation}
\label{eq3.8}
\begin{split}
\int_{-\infty}^{\xi}I_{*}(v)dv=\frac{d_{2}}{\gamma}I_{*}'(\xi)-\frac{c^{*}}{\gamma}I_{*}(\xi)+\frac{\beta}{\gamma}\int_{-\infty}^{\xi}\frac{S_{*}(v)I_{*}(v)}{S_{*}(v)+I_{*}(v)}dv,
\end{split}
\end{equation}
where we have used (\ref{eq3.1}) and (\ref{eq3.2}).  In view of (\ref{eq2.16}), (\ref{eq3.2}), (\ref{eq3.7}) and (\ref{eq3.8}), we have
\begin{equation}
\label{eq3.9}
\int_{-\infty}^{\infty}I_{*}(v)dv<\infty.
\end{equation}
Since $I_{*}'(\xi)$ is bounded in $\mathbb{R}$ (by (\ref{eq3.2})), one can obtain that
\begin{equation}
\label{eq3.10}
\begin{split}
I_{*}(\xi)\rightarrow0~~~\textmd{as}~~~\xi\rightarrow\infty.
\end{split}
\end{equation}
Moreover, from (\ref{eq3.2}), (\ref{eq3.7}), (\ref{eq3.8}) and (\ref{eq3.10}) we deduce that
\begin{equation}
\label{eq3.11}
\begin{split}
\int_{-\infty}^{\infty}I_{*}(v)dv=\frac{\beta}{\gamma}\int_{-\infty}^{\infty}\frac{S_{*}(v)I_{*}(v)}{S_{*}(v)+I_{*}(v)}dv=\frac{c^{*}(S_{-\infty}-S_{\infty})}{\gamma}.
\end{split}
\end{equation}
The proof is completed.
\end{proof}

\section{Positivity and upper bound of the critical traveling wave solution}
\setcounter {equation}{0}
In the section, we will derive the positivity and upper bound  of the critical traveling wave solution and finish the proof of Theorem \ref{the1.1}.
\begin{proposition}
\label{pro4.1}
Let $(S_*(\xi),I_*(\xi))$ be a critical traveling wave solution of (\ref{eq1.2}), then we have
\[
0<S_*(\xi)<S_{-\infty}\; \textmd{and}\;\ 0<I_*(\xi)<\min \bigg\{\frac{(\beta-\gamma)S_{-\infty}}{\gamma}, \frac{2c^*(S_{-\infty}-S_\infty)}{\sqrt{(c^*)^{2}+4d_2\gamma}+c^*}\bigg\}
\]
for $\xi\in\mathbb{R}$.
\end{proposition}
\begin{proof}
Firstly, it is easy to see from (\ref{eq3.1}), (\ref{eq3.4}) and (\ref{eq3.6}) that
\begin{equation}
\label{eq4.1}
S_{*}(\xi)<S_{-\infty} \;\; \textmd{for} \;\; \xi\in \mathbb{R}.
\end{equation}

Secondly, we claim that $S_{*}(\xi)>0$ for $\xi\in \mathbb{R}$. By contradiction, we assume that $S_{*}(\hat{\xi})=0$ for some $\hat{\xi}\in \mathbb{R}$. Then there exist two constants $l_1$, $l_2\in \mathbb{R}$ such that  $l_1<\xi_2\leq l_2$ and $\hat{\xi}\in (l_1, l_2)$. This implies that $S(\xi)$ attains its minimum in $(l_1, l_2)$. From the first equation in (\ref{eq2.15}) we get
\[
-d_1 S_*''(\xi)+c^*S_*'(\xi)+\beta S_*(\xi)\geq 0\;\;\textmd{for} \;\;\xi\in[l_1, l_2].
\]
 Then it follows from the strong maximum principle that
\[
S_{*}(\xi)\equiv0\;\; \textmd{for}\;\;  \xi\in[l_1, l_2],
\]
which contradicts the fact that $S_{*}(\xi)\geq \underline{S}(\xi)>0$ for $\xi\in[l_1, \xi_2)$.  Hence we have
\begin{equation}
\label{eq4.2}
S_{*}(\xi)>0 \;\; \textmd{for} \;\; \xi\in \mathbb{R}.
\end{equation}
In an analogous manner, one can get that
\begin{equation}
\label{eq4.3}
I_{*}(\xi)>0 \;\; \textmd{for} \;\; \xi\in \mathbb{R}.
\end{equation}

Thirdly, we assert that $I_*(\xi)<M=(\beta-\gamma)S_{-\infty}/\gamma$ for $\xi\in \mathbb{R}$. Assume by way of contradiction that $I_{*}(\tilde{\xi})=M$ for some $\tilde{\xi}\in \mathbb{R}$. Then we have $I_{*}'(\tilde{\xi})=0$ and $I_{*}''(\tilde{\xi})\leq0$. Evaluating the second equation in (\ref{eq2.15}) at the point $\tilde{\xi}$ and using (\ref{eq4.1}), we deduce that
\[
\begin{split}
0&=c^{*}I_*'(\tilde{\xi})\\
&=d_{2}I_*''(\tilde{\xi})+\frac{\beta S_*(\tilde{\xi})I_*(\bar{\xi})}{S_*(\tilde{\xi})+I_*(\tilde{\xi})}-\gamma I_*(\tilde{\xi})\\
&<\frac{\beta S_{-\infty}M}{S_{-\infty}+M}-\gamma M\\
&=0,
\end{split}
\]
which leads to a contradiction. Therefore we have
\begin{equation}
\label{eq4.4}
I_*(\xi)<\frac{(\beta-\gamma)S_{-\infty}}{\gamma} \;\; \textmd{for} \;\; \xi\in \mathbb{R}.
\end{equation}

Finally, we prove that
\[
I_{*}(\xi)<\frac{2c^{*}(S_{-\infty}-S_{\infty})}{\sqrt{(c^{*})^{2}+4d_{2}\gamma}+c^{*}} \;\;\textmd{for}\;\; \xi\in\mathbb{R}.
\]
To this end, we introduce an auxiliary function
\begin{equation}
\label{eq4.5}
\begin{split}
P(\xi)=I_{*}(\xi)+\frac{2\gamma}{\sqrt{(c^{*})^{2}+4d_{2}\gamma}+c^{*}}\int_{-\infty}^{\xi}I_{*}(v)dv.
\end{split}
\end{equation}
Then by (\ref{eq2.16}) and (\ref{eq3.9}), one can see that $P(\xi)$ is well-defined in $\mathbb{R}$. Since $I_*(\xi)\geq \underline{I}(\xi)>0$ on $(-\infty,\xi_3)$, we have $P(\xi)>0$ and $I_*(\xi)<P(\xi)$ for $\xi\in\mathbb{R}$. Note that $I_{*}(\pm\infty)=0$, then we get from
(\ref{eq4.5}) and (\ref{eq3.11}) that
\begin{equation}
\label{eq4.6}
P(-\infty)=0  \;\;\; \textmd{and} \;\;\; P(\infty)=\frac{2c^{*}(S_{-\infty}-S_{\infty})}{\sqrt{(c^{*})^{2}+4d_{2}\gamma}+c^{*}}.
\end{equation}
Differentiating (\ref{eq4.5}) with respect to $\xi$ yields
\begin{equation}
\label{eq4.7}
\begin{split}
P'(\xi)=I_{*}'(\xi)+\frac{2\gamma}{\sqrt{(c^{*})^{2}+4d_{2}\gamma}+c^{*}}I_{*}(\xi).
\end{split}
\end{equation}
Then from $I_{*}(\pm\infty)=I_{*}'(\pm\infty)=0$ we obtain that $P'(\pm\infty)=0$. Differentiating (\ref{eq4.7}) with respect to $\xi$ and using the second equation in (\ref{eq2.15}), we get
\begin{equation}
\label{eq4.8}
\begin{split}
-d_{2}P''(\xi)+bP'(\xi)&=-d_{2}I_{*}''(\xi)+c^{*}I_{*}'(\xi)+\gamma I_{*}(\xi)\\
&=\frac{\beta S_{*}(\xi)I_{*}(\xi)}{S_{*}(\xi)+I_{*}(\xi)},
\end{split}
\end{equation}
where
\[
b=\frac{c^{*}+\sqrt{(c^{*})^{2}+4d_{2}\gamma}}{2}.
\]

From (\ref{eq4.8}) and using $P'(\infty)=0$  and (\ref{eq2.16}), we infer that
\[
P'(\xi)=\frac{\beta}{d_{2}}\int_{\xi}^{\infty}e^{\frac{b}{d_{2}}(\xi-v)}\frac{S_{*}(v)I_{*}(v)}{S_{*}(v)+I_{*}(v)}dv\geq0,
\]
which together with (\ref{eq4.6}) yields
\[
P(\xi)\leq P(\infty)=\frac{2c^{*}(S_{-\infty}-S_{\infty})}{\sqrt{(c^{*})^{2}+4d_{2}\gamma}+c^{*}} \;\;\textmd{for}\;\; \xi\in\mathbb{R}.
\]
 Then from (\ref{eq4.5}) and (\ref{eq4.3}), we obtain that
\[
 I_{*}(\xi)< \frac{2c^{*}(S_{-\infty}-S_{\infty})}{\sqrt{(c^{*})^{2}+4d_{2}\gamma}+c^{*}} \;\;\textmd{for}\;\;  \xi\in\mathbb{R}.
\]
All the claims of this proposition are shown.
\end{proof}
\begin{remark}
\begin{itemize}
\item[(1)] In fact, the inequalities (\ref{eq3.4}), (\ref{eq4.2}) and (\ref{eq4.3}) imply that $S_*(\xi)$ is strictly decreasing in $\mathbb{R}$.
\item[(2)] When $\mathcal{R}_0>1$ and $c>c^{*}$, Wang et al. \cite{Wang-Wang-Wu-2012} introduced a function
\[
J(\xi):=I(\xi)+\frac{\gamma}{c}\int_{-\infty}^\xi I(y)dy+\frac{\gamma}{c}\int_{\xi}^{\infty} e^{-\frac{c}{d_2}(\xi-y)} I(y)dy, \quad \xi\in\mathbb{R}
\]
to obtain an upper bound of $I$-component in system (\ref{eq1.2}) with $c>c^*$, that is,
\[
I(\xi)\leq S_{-\infty}-S_{\infty}\;\; \textmd{for} \;\; \xi \in \mathbb{R}.
\]
While using the function $P(\xi)$ in (\ref{eq4.5}) by replacing $c^*$ by $c$, one can get a better estimate of $I$-component in (\ref{eq1.2}) with $c>c^*$, i.e.,
\[
I(\xi)<\frac{2c(S_{-\infty}-S_{\infty})}{\sqrt{c^{2}+4d_2\gamma}+c}\;\; \textmd{for} \;\; \xi \in \mathbb{R}.
\]
\end{itemize}
\end{remark}

\end{document}